\documentclass{amsart}
\usepackage[utf8]{inputenc}
\usepackage{amsmath}
\usepackage{amsthm}
\usepackage{amsfonts}
\usepackage{amssymb}
\usepackage{enumerate}
\usepackage{graphicx}
\usepackage{todonotes}
\usepackage{cite}
\usepackage{hyperref}
\usepackage{subfigure}
\graphicspath{ {./images/} }

\newcommand{\Prob}{\mathbb{P}}
\newcommand{\E}{\mathbb{E}}
\newcommand{\C}{\mathbb{C}}

\newcommand{\N}{\mathbb{N}}

\renewcommand\Im{\operatorname{Im}}
\newcommand{\eps}{\varepsilon}

\title{Rate of Convergence in Multiple SLE using Random Matrix Theory}
\author[A. Campbell]{Andrew Campbell}
	\address{Department of Mathematics\\ University of Colorado\\ Campus Box 395\\ Boulder, CO 80309-0395\\USA}
	\email{andrew.j.campbell@colorado.edu}
	
	\author[K. Luh]{Kyle Luh}
	\address{Department of Mathematics\\ University of Colorado\\ Campus Box 395\\ Boulder, CO 80309-0395\\USA}
	\email{kyle.luh@colorado.edu}
	\thanks{K. Luh was supported in part by the Ralph E. Powe Junior Faculty Enhancement Award.}
	
	\author[V. Margarint]{Vlad Margarint}
    \address{Department of Mathematics\\ University of Colorado\\ Campus Box 395\\ Boulder, CO 80309-0395\\USA}
	\email{vlad.dumitrumargarint@colorado.edu}

\theoremstyle{plain}
\newtheorem{theorem}{Theorem}[section]

\newtheorem{proposition}[theorem]{Proposition}

\newtheorem{lemma}[theorem]{Lemma}

\theoremstyle{definition}
\newtheorem{definition}[theorem]{Definition}

\newtheorem{remark}[theorem]{Remark}

\newcommand{\RR}{\mathbb{R}}
\newcommand{\CC}{\mathbb{C}}

\newcommand{\EE}{\mathbb{E}}

\DeclareMathOperator{\tr}{tr}

\begin{document}

\maketitle

\begin{abstract}
We provide an order of convergence for a version of the Carath\'eodory convergence for the multiple SLE model with a Dyson Brownian motion driver towards its hydrodynamic limit, for $\beta=1$ and $\beta=2$. The result is obtained by combining techniques from the field of Schramm-Loewner Evolutions with modern techniques from random matrices. Our approach shows how one can apply modern tools used in the proof of universality in random matrix theory, in the field of Schramm-Loewner Evolutions. 
\end{abstract}

% \tableofcontents
% \newpage

\section{Introduction and main results}
Schramm-Loewner Evolution (SLE) and random matrix theory (RMT) are two active and well-studied fields of research within modern
probability theory \cite{lawler2008conformally, akemann2011oxford}. The SLE was introduced by Oded Schramm in 2000 in his study of scaling limits of various discrete processes \cite{schramm2000scaling}.  RMT appeared earlier in the statistical work of Wishart \cite{wishart1928generalised} and the pioneering physics of Wigner \cite{wigner1955}. Both SLE and RMT have been thriving areas of mathematical research since their advent. 

When studying SLE theory, one introduces the notion of compact hulls, which are compact sets with simply connected complements in the upper half-plane.
If $K_t$ is a growing set of hulls parameterized by $t \in [0, T]$ and the growth is local in some sense, then it is known that $g_t:= g_{K_t}$ obeys the Loewner differential equation
\[
\partial_t g_t(z) =  \frac{2}{g_t(z) - W_t}
\]
where $W_t$ is referred to as the driving function and captures the local growth of $K_t$. 
SLE are the random curves corresponding to the $g_t$ when the driving function is a constant multiple of Brownian motion, that we denote by $\sqrt{\kappa}B_t$,  for $\kappa \geq 0$. With probability one, $g_t$ is continuous up to the boundary and the limit
$$
\gamma(t)=\lim_{y\to 0} g_t^{-1}(\sqrt{\kappa}B_t+iy)
$$
exists and is continuous in time, by the Rohde-Schramm Theorem \cite{RohdeSchramm}. The curve $\gamma(t)$ is called the SLE trace.
Also, it can be shown that with probability one, $g_t$ is a continuous family of conformal maps from $H_t$ to $\mathbb{H}$, where $H_t$ is the unbounded component of the complement in $\mathbb{H}$ of $\gamma(t),$ for $t \in [0, T]$ \cite{RohdeSchramm}.
Moreover, the nature of the curve changes as $\kappa$ increases from simple a.s.  when $\kappa \in [0,4]$, to having double points a.s. for $\kappa \in (4,8)$ and space-filling a.s., for $\kappa \geq 8$.
For different parameters $\kappa$, the SLE models the scaling limits of an astoundingly diverse set of discrete models. For instance, it was proved in \cite{Lawlerschrammwerner}  that the scaling limit of the loop erased random walk (with the loops erased in a chronological order) converges in the scaling limit to $\text{SLE}_{\kappa}$ with  $\kappa = 2.$ Moreover, other two dimensional discrete models from Statistical Mechanics including the Ising model cluster boundaries, Gaussian free field interfaces, percolation on the triangular lattice at critical probability, and Uniform spanning trees were proved to converge in the scaling limit to SLE for values of $\kappa=3,$ $\kappa=4,$ $\kappa=6$ and $\kappa=8$ respectively in the series of works \cite{Stasising}, \cite{SLEGFF}, \cite{Smirnovpercolation}  and \cite{Lawlerschrammwerner}.

One can consider more generally the Loewner equation driven by a time-dependent real-valued measure $\mu_{t}$
$$
\frac{\partial}{\partial t} g_{t}(z)=\int_{\mathbb{R}} \frac{\mu_{t}(d x)}{g_{t}(z)-x}, \quad g_{0}(z)=z.
$$

%Throughout the paper, we use the notation $g_t^{\infty}(z)$ for the multiple SLE maps when the driving measure is the semicircle law. 

When the driving measure $\mu_{t}$ is a Dirac-delta mass at location $\sqrt{\kappa}B_t$, we recover the previous SLE maps. In the case $\mu_{t}=\sum_{i=1}^{N} \omega_{i}(t) \delta_{U_{i}(t)}$, for some non-intersecting continuous functions $U_{i}(t) \in \mathbb{R}$ (called driving functions), and weights $\omega_{i}(t) \in \mathbb{R}^{+},$ we obtain the multi-slit Loewner equation with driving functions $U_i(t)$, $i=1,\cdots ,n$. In this work, we consider the case $\omega_{i}(t)=1/N$, for all $t \in [0, T].$ 

For a real parameter $\beta >0$, Dyson Brownian motion (DBM) is defined by the following system of N equations
\begin{equation}\label{dbm}
d\lambda^{(i)}_t=\frac{2}{\sqrt{N\beta}}dB_t^{(i)}+\frac{2}{N}\sum_{j \neq i}\frac{dt}{\lambda_t^{(j)}-\lambda_t^{(i)}},
\end{equation}
 for $i=1,2,...,N$.

Due to its connections with other fields, an important Loewner equation is the multiple SLE with DBM as a driver. The multiple SLE maps that are obtained when the driving measure is an empirical measure on $N$ DBM particles are denoted in this paper by $g_t^N(z).$  This model was introduced by Cardy in \cite{cardy2003stochastic}, and studied further by Lawler and Healey in \cite{lawlernew}, in connection with the quantum Calogero-Sutherland model and Conformal Field Theory. More works on the connection between Multiple SLE and CFT can be found in \cite{LenVik} and \cite{MultipleSLEcft}. In the case of $N=2$ curves, perturbations of this model in the parameter $\beta$ have been studied in \cite{JVperturbation}. We note that the parameters $\beta$ in the DBM model and $\kappa$ in SLE theory are related via $\beta=8/\kappa$.

We refer to the multiple SLE model with Dyson Brownian motion as a driver as the simultaneously growing multiple SLE model. There is also a version of the multiple SLE that has non-simultaneous growth that has received a lot of attention in the previous years. There have been several results on the multiple SLE model in both the upper half-plane and the unit disk versions \cite{Katoriwelding, LenVik, PeltolaWu, BefPeltolaWu, HotaS, hydrodinamiclimit, delMonacoetc, delMonaco2, dubedat, peltolakyto, oliversc, zh1, zh2, multipleSLE0}.

In \cite{delMS2016}, the authors consider the $N\rightarrow\infty$ limit of multiple SLE driven by DBM. In particular they show that the empirical measure of the initial positions converges to a probability measure $\mu_0$, then $g_t^N$ converges in distribution with respect to locally uniform convergence to $g_t^\infty$ solving \begin{equation}\label{eq:A:g infty def}
    \frac{\partial}{\partial t}g_t^\infty(z)=M_t^\infty(z),\qquad g_0(z)=z,
\end{equation} Where $M_t^\infty$ is a solution to the complex Burgers equation\begin{equation}\label{eq:A:Complex Burgers Equation}
         \begin{cases}
                  \frac{\partial M_t^\infty(z)}{\partial t}=-2M_t^\infty(z)\frac{\partial M_t^\infty(z)}{\partial z},\ t>0,\\
                  M_0^\infty(z)=\int_\RR\frac{2}{z-x}d\mu_0(x).
           \end{cases}
           \end{equation}
           Their result serves as the multiple SLE analog of Wigner's famous semicircle law in random matrix theory. We consider this model and we obtain more refined information by providing an order of convergence of this model in a weaker version of the Carath\'eodory type convergence. We aim in future works to study the full Carath\'eodory convergence by strengthening the estimates as we approach the multiple SLE hull.

% As the Dyson Brownian motion along with the local statistics of Random Matrices ensembles are actively studied, this gives many possible avenues of research between Multiple SLE and Random Matrix Theory in the upcoming years.

In this work, we combine elements of the proof of Local Laws in random matrix theory, such as resolvent techniques, with elements of the SLE theory.  In other words, we apply modern techniques from random matrix theory to the analysis of SLE. 
 Local Laws are a very important research direction in random matrix theory in the last years (see \cite{Antti}, \cite{13}, \cite{14}, \cite{15}, \cite{16}, \cite{17}, \cite{18}, \cite{19}, \cite{20}, \cite{21}, \cite{22}, \cite{23}, \cite{VladRMT}, \cite{SeanVu}, \cite{32}, \cite{33}, \cite{34}, \cite{35}, \cite{36} for a non-exhaustive list). They are one of the fundamental ingredients in proving the universality of Wigner ensembles in random matrix theory (see \cite{erdos2017dynamical}). %One method to prove such local laws is to perform a concentration of measure type bound along with a stability result for the self-consistent equations. 

Given the outstanding developments in the proof of the universality in RMT using the analysis of the DBM, the interaction between multiple SLE and random matrices will provide many avenues to explore. The approach in the current work represents one of the possible directions of exploration between these two major fields of Probability theory. In a different direction, which we aim to explore in the future, one can study the geometry of the multiple SLE curves using the analysis of the Dyson Brownian motion drivers, as well as good approximation schemes of the model (see, for example, \cite{Jamespaper}, \cite{tranconvergence}, \cite{NVscheme} in the one SLE curve case). Yet another possibility is to study the continuity of the multiple SLE model in the parameter $\beta$, motivated by the great interest and progress throughout the years in this yet unresolved conjecture in the one SLE curve case (see \cite{BLM}, \cite{AtulVladpaper}, \cite{FrizYuan}, \cite{SLEcont}). In addition, the fact that the multiple SLE curves grow from the positions of the drivers along with some knowledge about the structure of the drivers gives the possibility of defining new observables in order to study the convergence of discrete models to the multiple SLE. Examples of such observables include the statistics of the $k^{th}$ smallest distance between drivers, for $k \geq 1$, (see \cite{BAB}) or the probability of having no drivers in a symmetric region about the origin (see \cite{Noeigen} for $\beta=2$).

 Although our result can be obtained for general bounded initial conditions, we state it in the case in which all the Dyson Brownian motion particles start from the origin. We prefer this choice for the simplicity of the notation and exposition.

\begin{theorem}\label{thm: Main Result}
Let $\beta=1$ or $\beta=2$, and let us consider Dyson Brownian motion beginning at the origin. Let $K_T$ be the multiple SLE hull at time $T>0.$ Then, for any $\eps >0$, for the multiple SLE maps for $N$ curves, we have that 
$$\sup_{t \in [0,T],\ z \in G}|g^N_t(z)-g^{\infty}_t(z)|=O\left(\frac{1}{N^{1/3-\eps}}\right),$$
with overwhelming probability\footnote{An event $E$ holds with overwhelming probability if, for every $p > 0$,  $\Prob(E) \geq 1 - O_p(n^{-p})$; see Definition \ref{def:events} for details.}, for a given $G \subset \mathbb{H}\setminus K_T.$
\end{theorem}

\begin{remark}
 It is well-known that for the special values of the parameters $\beta=1$, $\beta=2$ and $\beta=4$, the Dyson Brownian motion particles statistics can be understood using matrices as these values correspond to the well-studied models of the Gaussian Orthogonal Ensemble, Gaussian Unitary Ensemble (GUE), and the Gaussian Symplectic Ensemble (GSE) respectively. An $n \times n$ real symmetric matrix $A$ is drawn from the Gaussian Orthogonal Ensemble (GOE) if the upper-triangular entries $A_{ij}$, $1 \leq i \leq j \leq n$ are independent Guassian random variables, where $A_{ij}$ has mean zero and variance $\frac{1 + \delta_{ij}}{n}$ and $\delta_{ij}$ is the Kronecker delta. The GUE and GSE ensembles are defined similarly with complex and quaternic Gaussian off-diagonal entries. We study the cases $\beta=1$ and $\beta=2$ respectively as they correspond to the critical parameters $\kappa=8$ and $\kappa=4$ in SLE theory.  We expect that a similar analysis will hold for the case $\beta=4$ that corresponds to the value $\kappa=2$.

We note that the $N^{-(1/3-\epsilon)}$ order of convergence to the hydrodynamic limit of multiple SLE is obtained via an estimate in \cite{SeanVu} which is, to the best of our knowledge, the best stability estimate in this setting available in the literature. 

Theorem \ref{thm: Main Result} relies on the following technical result. \begin{theorem}\label{thm:Local law theorem}
    Let $\beta=1$ or $\beta=2$, and let us consider Dyson Brownian motion started from the origin $\left(\lambda_t^{(1)},\dots,\lambda_t^{(N)}\right)$ and $M_t^N:\CC_+\rightarrow \CC_-$ defined by\begin{equation*}
        M_t^N(z)=\frac{1}{N}\sum_{j=1}^N\frac{2}{z-\lambda_t^{(j)}}.
    \end{equation*}  Let $M_t^\infty:\CC_+\rightarrow \CC_-$ be the solution to the complex Burgers equation \begin{equation}\label{eq:Complex Burgers Equation}
         \begin{cases}
                  \frac{\partial M_t^\infty(z)}{\partial t}=-2M_t^\infty(z)\frac{\partial M_t^\infty(z)}{\partial z},\ t>0,\\
                  M_0^\infty(z)=\frac{2}{z}.
           \end{cases}
    \end{equation} Then for any compact set $G\subset \CC_+$, $\eps > 0$, and fixed $t\in [0,T]$ \begin{equation}
        \sup_{z\in G}\left|M_t^N(z)-M_t^\infty(z) \right|=O_{G,\eps}\left( \frac{t}{N^{\frac{1}{3}-\eps}} \right),
    \end{equation} with overwhelming probability. 
\end{theorem}

\end{remark}

\vspace{2mm}

 The remainder of the paper is organized into several sections. In the second section, we present probabilistic estimates involving the multiple SLE hull and subsets of its complement. The third section focuses on the random matrix techniques we use, as well as on the proof of Theorem 1.3. In subsection 3.4 we utilize a net argument that extends the previously obtained results for a fixed time $t \in [0, T]$, to all times simultaneously. In section 4, we prove Theorem 1.1 and in the Appendix we provide the stability part of the argument. 
 
 % along with an estimate for general $\beta \geq 1$ that we plan to use in future work where we aim to show the result beyond the regimes $\beta=1$ and $\beta=2$.

\vspace{5mm}

\section{Subset of the complement of the multiple SLE hull}

\vspace{5mm}

In this section, we provide probabilistic estimates for general $\beta \geq 1$ that are useful in deducing the choice of the set $G \subset \mathbb{H}\setminus K_T$, where we establish the order of convergence of the family of maps. We present the estimates for general $\beta \geq 1$, and specialize to the $\beta=1$ and $\beta=2$ cases in our application.

Let $\partial_t g_t(z)=\frac{1}{N} \sum_{i=1}^N \frac{2}{g_t(z)-\lambda^i_t}, $ where $(\lambda_t^{(1)}, \cdots, \lambda_t^{(N)})$ is a Dyson Brownian motion (DBM) with parameter $\beta \geq 1.$
We first consider $\lambda^i_t \equiv 0, \forall t \in[0, T]$, for all $i = \{ 1, 2, \cdots, N\}.$ Then, we have that $\partial_t g_t(z)=\frac{2N}{N g_t(z)}=\frac{2}{g_t(z)}.$ Since $g_t(z)=\text{Re}(g_t(z))+i\text{Im}(g_t(z)),$ 
we have that $$\partial_t \text{Im}(g_t(z))=\frac{-2 \text{Im}( g_t(z))}{\left|g_{t}(z)\right|^2} \geq \frac{2}{(\text{Im}(g_t(z))^2}.$$
This allows us to conclude that $$\text{Im}\left(g_t(z)\right)^2 \geqslant\left(\text{Im}(z)\right)^2-4 t>0,$$
whenever $\text{Im}(z)>2\sqrt{T}.$

In order to control the real part, for a Dyson Brownian motion $(\lambda_t^{(1)}, \cdots, \lambda_t^{(N)})$  with parameter $\beta \geq 1$, we observe that 
 $$\partial_t \text{Re}(g_t(z))=\frac{1}{N}\sum_{i=1}^N \frac{\text{Re}\left(g_t(z)\right)-\lambda_t^i}{\left|g_t(z)-\lambda^i_t\right|^2}>0,$$ whenever $\text{Re}(g_t(z))>M=\sup _{t \in[0, T]}\sup_{i=\{1, 2, \ldots, N\}}\left|\lambda_t^i\right|.$
Then, combining the two estimates, we have that $$\{z \in H|:| \text{Re}(z)>M \hspace{1mm} \text{or} \hspace{1mm} \text{Im} > 2\sqrt{T} \} \subset \mathbb{H}\setminus K_T.$$

We also note that for all $t \in [0, T]$, we have

$$K_t \subset\{z \in \overline{\mathbb{H}}:|\operatorname{Re} z| \leq M \hspace{1mm} and \hspace{1mm} \operatorname{Im} z \leq 2 \sqrt{T}\}.$$

Next, we use the following probabilistic result on the behaviour of the extreme eigenvalues.
\begin{lemma}[Lemma 4.3.17 in \cite{Zeit}]\label{Zeitlemma}
Let $\lambda_N^*(t):=\max _{1 \leq i \leq N}\left|\lambda^{(i)}_t\right|=\max \left(\lambda_t^{(N)},-\lambda_t^{(1)}\right).$ Let $\beta \geq 1$. Then there exist finite constants $\alpha=\alpha(\beta)>0, C=C(\beta)$, and for all $t \geq 0$ a random variable $\eta_N^*(t)$ with law independent of $t$, such that
$$
P\left(\eta_N^*(t) \geq x+C\right) \leq e^{-\alpha N x}
$$ and, for all $t \geq 0$,
$$
\lambda_N^*(t) \leq \lambda_N^*(0)+\sqrt{t} \eta_N^*(t).
$$
\end{lemma}

In the case of the DBM drivers, using Lemma \ref{Zeitlemma}, we have that for $\beta \geq 1$ and for $C=C(\beta)$ and $\alpha=\alpha(\beta)$ some finite constants that $$\mathbb{P}\left( \sup _{t \in[0, T]}\sup_{i=\{1, 2, \ldots, N\}}\left|\lambda_t^i\right| \leq (C+x)\sqrt{T}\right) \geq 1-e^{-\alpha N x}.$$

For conformal maps, we have the following result.

\begin{lemma}[Lemma $4.5$ in \cite{Kemp}] Let $K$ be a hull and $H=\mathbb{H} \backslash K$. If $K \subset B\left(x_0, r\right)$, then $g_K$ maps $H \cap B\left(x_0, 2 r\right)$ into $B\left(x_0, 3 r\right)$ and $$\sup _{z \in H}\left|g_K(z)-z\right| \leq 5 r.$$
\end{lemma}

For a box $G \subset H_T=\mathbb{H}\setminus K_T$, we have that with overwhelming probability that

\begin{equation}\label{eq:A:set containment}
g^N_t(G) \subset \{z: \sqrt{\text{Im}(z_0))^2-4 t} \leq \text{Im}(z) \leq  \text{Im}(z_0); \hspace{2mm} |Re(z)| \leq f(N, T)\},
\end{equation}
where $f(N,T)$ can be deduced from the following:

\begin{equation}
|\text{Re}g_K(z)| \leq |g_K(z)| \leq |z|+5r.
\end{equation}
In the case of the multiple SLE hull $K_T$, we have $r=\sqrt{M^2+(2\sqrt{T})^2}$.

\section{Random Matrix Techniques} In this section we prove some random matrix results leading to the proof of Theorem \ref{thm:Local law theorem}. It is worth noting that for $\beta=1$ and $\beta=2$, DBM $\left(\lambda_t^{(1)},\dots,\lambda_t^{(N)}\right)$ defined as the solution to \eqref{dbm} starting from initial positions $\left( \lambda_0^{(1)},\dots, \lambda_0^{(N)} \right)$ is equal in distribution to the eigenvalues of $D-2\sqrt{t}A$ where $D$ is an $N\times N$ diagonal matrix of the initial positions and $A$ is a matrix drawn from the Gaussian Orthogonal Ensemble for $\beta=1$ or Gaussian Unitary Ensemble (GUE) for $\beta=2$. We establish the results in this section for the case when $A$ is drawn from the GOE, since the adjustments to the GUE model are straightforward. 

\subsection{Tools}

    	This section introduces the tools we will use throughout.  We begin with a definition describing high probability events.  
	
	\begin{definition}[High probability events] \label{def:events}
		Let $E$ be an event that depends on $n$.
		\begin{itemize}
			\item $E$ holds \emph{asymptotically almost surely} if $\Prob(E) = 1 - o(1)$.
			\item $E$ holds \emph{with high probability} if $\Prob(E) = 1 - O(n^{-c})$ for some constant $c > 0$.
			\item $E$ holds \emph{with overwhelming probability} if, for every $p > 0$,  $\Prob(E) \geq 1 - O_p(n^{-p})$.  
		\end{itemize}
	\end{definition}
	
	For $z = E + i \eta \in \mathbb{C}_+$, $n\times n$ Hermitian matrix $H$, and $G(z):=\left(H-zI\right)^{-1}$ the \emph{Ward identity} states that
	\begin{equation} \label{eq:ward}
		\sum_{j = 1}^n \left| G_{ij}(z) \right|^2 = \frac{1}{ \eta} \Im G_{ii}(z). 
	\end{equation} 
	
	If $A$ and $B$ are invertible matrices, the \emph{resolvent identity} states that
	\begin{equation} \label{eq:resolvent}
		A^{-1} - B^{-1} = A^{-1} (B - A) B^{-1} = B^{-1} (B - A) A^{-1}. 
	\end{equation} 
	
	If $\xi$ is a Gaussian random variable with mean zero and variance $\sigma^2$ and $f: \mathbb{R} \to \mathbb{C}$ is continuously differentiable, the \emph{Gaussian integration by parts formula} states that
	\begin{equation} \label{eq:ibp}
		\E[ \xi f(\xi)] = \sigma^2 \E[ f'(\xi) ], 
	\end{equation}
	provided the expectations are finite.  
	The next lemma is a convenient moment bound for a martingale difference sequence. 
	\begin{lemma} [Lemma 2.12 from \cite{kyle8}] \label{klem:burkholder2}
		Let $\{X_k\}$ be a complex martingale difference sequence and $\mathcal{F}_k = \sigma(X_1, \dots, X_k)$ be the $\sigma$-algebra generated by $X_1, \dots, X_k$.  Then, for any $p \geq 2$, 
		\[
		\E  \left|\sum_{k=1}^n  X_k \right|^p \leq C_p \left(\E\left( \sum_{k=1}^n \E_{k-1} |X_k|^2 \right)^{p/2} + \E \sum_{k=1}^n |X_k|^p \right). 
		\]
		where $C_{p}$ is a constant that only depends on $p$ and $\E_{k-1}[\cdot] := \E[\cdot | \mathcal{F}_{k-1}]$.  
	\end{lemma}
	The next concentration lemma is helpful in controlling the deviation of a quadratic form from its expectation. 
	\begin{lemma} [Equation (3) from \cite{kyle2}] \label{klem:quadraticform}
		Let $X$ be an $n$-vector containing iid standard Gaussian random variables, $A$ a deterministic $n \times n$ matrix and $\ell \geq 1$ an integer.  Then
		\[
		\E[X^* A X - \tr A|^{2 \ell} \leq C_\ell (\tr A A^* )^\ell
		\]
		where $C_\ell$ is a constant that only depends on $\ell$.  
	\end{lemma}
	Finally, we will require the following algebraic identity in Section \ref{Concentration of Stieltjes transform}.
	\begin{lemma} [Theorem A.5 from \cite{kyle8}] \label{klem:tracedifference}
		Let $A$ be an $n \times n$ symmetric matrix and $A_k$ be the $k$-th major submatrix of size $(n-1) \times (n-1)$.  If $A$ and $A_k$ are both invertible, then
		\[
		\tr( A^{-1}) - \tr(A_k^{-1}) = \frac{1+ \alpha_k^* A_k^{-2} \alpha_k}{A_{kk} - \alpha_k^* A_k^{-1} \alpha_k}
		\] 
		where $\alpha_k$ is obtained from the $k$-th column of $A$ by deleting the $k$-th entry.  
	\end{lemma}

	\subsection{Concentration of the Gaussian Orthogonal Ensemble}\label{Concentration of Stieltjes transform}

	In this section we show that $|M_{t}^N(z) - \E M_{t}^N(z)|$ is small for a fixed $z \in \C_+$. To match the random matrix literature we will consider for fixed $t>0$, $m_N(z):=-\frac{1}{2}M_t^N(z)$. We let $A_t$ be $\sqrt{t} A$ where $A$ is drawn from the Gaussian Orthogonal Ensemble.

 We note that $m_N(z) - \E m_N(z)$ can be written as the following telescopic sum
	\[
	m_N(z) - \E m_N(z) = \sum_{k=1}^n \left( \E_{k} m_N(z) - \E_{k-1} m_N(z) \right) := \sum_{k=1}^N \gamma_k
	\]
	
	Observe that
	\[
	m_N(z) = \frac{1}{N} \tr (A_t - z)^{-1} = \frac{1}{N} \tr \frac{1}{\sqrt{t} A - z} =\frac{1}{N \sqrt{t}} \tr \frac{1}{ A - z/\sqrt{t}} = \frac{1}{N \sqrt{t}} \tr \frac{1}{ A - z'}
	\]
	
	We define $E' = E/\sqrt{t}$ and $\eta' = \eta/\sqrt{t}$.  Let $\E_k$ denote the conditional expectation with respect to the $\sigma$-field generated by $A_{ij}$ with $i,j \leq k$, so that $\E_N m_N(z) = m_N(z)$ and $\E_0 m_N(z) = \E m_N(z)$. 
	
	\begin{align*}
		\gamma_k &= \frac{1}{N \sqrt{t}}(\E_{k} \tr ( A -z')^{-1} - \E_{k-1} \tr (A -z')^{-1} )  \\
		&= \frac{1}{N \sqrt{t} } \Big(\E_{k} \big[ \tr ( A -z')^{-1} - ( A_k - z')^{-1} \big] - \E_{k-1} \big[ \tr ( A-z')^{-1} - \tr ( A_k - z')^{-1} \big] \Big) \\
		&= \frac{1}{N \sqrt{t}} (\E_{k} - \E_{k-1}) \Bigg(\frac{a_k^* G_k^{2} a_k - \E_{a_k} a_k^* G_k^{2} a_k}{A_{kk} - z' - a_k^* G_k a_k} + \frac{1 + \E_{a_k} a_k^* G_k^{2} a_k }{A_{kk} - z' - a_k^* G_k a_k} \\
		&\hspace{7cm}- \frac{1 + \E_{a_k} a_k^* G_k^{2} a_k}{A_{kk} - z' - \E_{a_k} a_k^* G_k a_k} \Bigg) \\
		&= \frac{1}{N \sqrt{t}} (\E_{k} - \E_{k-1}) \Bigg(\frac{a_k^* G_k^{2} a_k - \E_{a_k} a_k^* G_k^{2} a_k}{A_{kk} - z' - a_k^* G_k a_k} \\
		&\hspace{4cm} - \frac{(1 + \E_{a_k}a_k^* G_k^{2} a_k) (a_k^* G_k a_k - \E_{a_k} a_k^* G_k a_k)}{(A_{kk} - z' -  a_k^* G_k a_k)(A_{kk} - z' - \E_{a_k} a_k^* G_k a_k)}\Bigg) \\
		&= \frac{1}{N \sqrt{t}} (\E_{k} - \E_{k-1}) \Bigg(\frac{a_k^* G_k^{2} a_k - \frac{1}{N} \tr G_k^{2} }{A_{kk} - z' - a_k^* G_k a_k} \\
		&\hspace{4cm} - \frac{(1 + \frac{1}{N} \tr G_k^{2} ) (a_k^* G_k a_k - \frac{1}{N}\tr G_k )}{(A_{kk} - z' -  a_k^* G_k a_k)(A_{kk} - z' - \frac{1}{N} \tr G_k )}\Bigg) \\
	\end{align*}
where $a_k$ denotes the $k$-th row of $A$ with the $k$-th entry removed.  
We define the following quantities,
\[
\alpha_k = a_k^* G_k^{2} a_k - \frac{1}{N} \tr G_k^{2}, 
\]
\[
\beta_k = \frac{1}{ A_{kk} - z' -  a_k^* G_k a_k}, \quad \bar{\beta}_k = \frac{1}{A_{kk} - z' - \frac{1}{N} \tr G_k}, 
\]
\[
\delta_k = a_k^* G_k a_k - \frac{1}{N}\tr G_k, \quad 
\epsilon_k = 1 + \frac{1}{N} \tr G_k^{2} ,
\]
so that
\begin{align} \label{eq:concentration}
	m_N(z) - \E m_N(z) &= \frac{1}{N \sqrt{t} } \sum_{k=1}^N (\E_{k} - \E_{k-1}) \alpha_k \beta_k  - \frac{1}{N \sqrt{t}} \sum_{k=1}^N (\E_{k} - \E_{k-1}) \epsilon_k \delta_k \beta_k \bar{\beta}_k \nonumber \\
	&:=  \frac{1}{\sqrt{t}} S_1 - \frac{1}{\sqrt{t}} S_2.
\end{align}

For a fixed $\varepsilon > 0$, we will show that $N^{1-\varepsilon} (\eta')^3  |S_1| = o(1)$ and $N^{1 - \varepsilon} (\eta')^3  |S_2| = o(1)$  with overwhelming probability.  This will be done via the method of moments.  We begin with $S_1$.  By Markov's inequality, it suffices to bound $\E|  N^{1- \varepsilon} (\eta')^3 S_1|^{2 \ell} = \E | N^{-\varepsilon} (\eta')^3\sum_{k =1}^n (\E_{k} - \E_{k-1}) \alpha_k \beta_k|^{2 \ell}$ for $\ell \in \N$.  

By Lemma \ref{klem:burkholder2}, for any $\ell \geq 1$,  
\begin{align*}
\E |N^{-\varepsilon} (\eta')^3 \sum_{k =1}^N (\E_{k} - \E_{k-1}) \alpha_k \beta_k|^{2 \ell} &\leq C_\ell \Bigg( \E \left(\sum_{k=1}^N \E_{k-1} |N^{-\varepsilon} (\eta')^3 \alpha_k \beta_k|^2 \right)^{\ell} \\
&\quad \quad + \sum_{k=1}^N \E  |N^{-\varepsilon} (\eta')^3 \alpha_k \beta_k|^{2 \ell} \Bigg).
\end{align*}
We use $C_\ell$ to indicate a constant that only depends on $\ell$, but may change from line to line.  Since $\Im a^*_k G_k a_k > 0$,
\[
|\beta_k| \leq (\eta')^{-1}.
\]  
Therefore,
\begin{align} \label{keq:S1}
	\E \left|N^{- \varepsilon}(\eta')^3 \sum_{k=1}^n (\E_{k} - \E_{k-1}) \alpha_k \beta_k \right|^{2 \ell} &\leq C_\ell N^{-2 \varepsilon \ell} \Bigg( \E \left(\sum_{k=1}^N \E_{k-1} |(\eta')^2 \alpha_k |^2 \right)^{\ell} \nonumber \\
	&\qquad \qquad + \sum_{k=1}^N \E  | (\eta')^2 \alpha_k|^{2 \ell} \Bigg).
\end{align}
By Lemma \ref{klem:quadraticform}, 
\[
\E|(\eta')^2 \alpha_k|^{2 \ell} \leq C_\ell (\eta')^{4 \ell} N^{-2 \ell}  \E|\tr G_k^{2} G_k^{*2}|^\ell.
\]
We use the simple bound that 
\begin{align} \label{keq:indicator}
	\tr G_k^2 G_k^{*2} &=  \left(\sum_{i=1}^N \frac{1}{((\lambda_i - E)^2 + (\eta')^2)^2} \right) \nonumber \\
	&\leq N (\eta')^{-4}
\end{align}

We now have that
\begin{align*}
	\E|(\eta')^2 \alpha_k|^{2 \ell} &\leq C_\ell (\eta')^{4 \ell} N^{-2 \ell} \E | N (\eta')^{-4}|^\ell \\
	&\leq C_\ell N^{- \ell}  
\end{align*}

Therefore, by equation \eqref{keq:S1},
\begin{equation*}
	\E \left|N^{-\varepsilon} (\eta')^3 \sum_{k=1}^N (\E_{k} - \E_{k-1}) \alpha_k \beta_k \right|^{2 \ell} \leq  C_\ell N^{-2 \varepsilon \ell} \left( \E \left(\sum_{k=1}^N \E_{k} |(\eta')^2 \alpha_k|^2 \right)^{\ell} + N^{-\ell+1} \right) 
\end{equation*}

By the same reasoning as in \eqref{keq:indicator}, we also have that $\E_k |\alpha_k|^2 \leq N (\eta')^{-4}$.  Thus,
\[
\E_{k} |(\eta')^2 \alpha_k|^2  \leq K N^{-1} 
\]
so
\[
 \E \left(\sum_{k=1}^N \E_{k} |(\eta')^2 \alpha_k|^2 \right)^{\ell} \leq C_\ell.
\]
Finally, we can conclude that 
\[
\E \left|N^{-\varepsilon} (\eta')^3 \sum_{k=1}^N (\E_{k-1} - \E_k) \alpha_k \beta_k \right|^{2 \ell} \leq C_\ell N^{-2 \varepsilon \ell}.
\]
As $\ell$ is arbitrary, we have shown that $|S_1| = o_{\eta} (t^{3/2}/N^{1 - \varepsilon})$ with overwhelming probability.  

Now we address $S_2$.  We first observe that
\begin{align*}
\left|1 + \frac{1}{N} \tr G_k^2\right| &\leq 1 + \frac{1}{N} \tr G_k G_k^* \\
&= (\eta')^{-1} \Im\left(-A_{kk} + z' + \frac{1}{N} \tr G_k \right)
\end{align*}
Therefore,
\begin{align*}
	|\epsilon_k \bar{\beta}_k| = \frac{|1 + \frac{1}{N} \tr G_k^2|}{|A_{kk} - z' - \frac{1}{N} \tr G_k|} \leq (\eta')^{-1}
\end{align*}
Recalling that $|\beta_k| \leq (\eta')^{-1}$, we have that
\[
\E |N^{1 - \varepsilon} (\eta')^4 S_2|^{2 \ell} = N^{-2 \varepsilon \ell} (\eta')^{2 \ell} \left|  \sum_{k=1}^N (\E_{k} - \E_{k-1})  \delta_k \right|^{2 \ell}
\]
Again, by Lemma \ref{klem:burkholder2}
\begin{align*}
	\E |N^{1 - \varepsilon} (\eta')^4 S_2|^{2 \ell} &\leq C_\ell N^{-2 \varepsilon \ell} (\eta')^{2\ell} \left( \E  \left( \sum_{k=1}^N \E_{k-1} |\delta_k|^2\right)^{\ell} + \sum_{K=1}^N \E|\delta_k|^{2 \ell}\right).
\end{align*}	
Note that by Lemma \ref{klem:quadraticform},
\[
\E |\delta_k|^{2 \ell} \leq C_\ell N^{-2 \ell} \E|\tr G_k G_k^*|^{\ell}.
\]
We have that
\[
\tr G_k G_k^* \leq N (\eta')^{-2}
\]
so 
\[
\E |\delta_k|^{2 \ell} \leq C_\ell N^{-\ell} (\eta')^{-2 \ell}. 
\]
Additionally,
\[
\E_{k-1} |\delta_k|^2 \leq N^{-1} (\eta')^{-2}.
\]
Thus,
\[
	\E |N^{1 - \varepsilon} (\eta')^4 S_2|^{2 \ell} \leq C_\ell N^{-2 \varepsilon \ell}.
\]
We can then conclude that $S_2$ is $o_\eta (t^2/N^{1- \varepsilon})$ with overwhelming probability. Returning to $\eqref{eq:concentration}$ we have shown that 
\begin{equation}\label{eq:A:Concentration Conclusion}
  |m_N(z) - \E m_N(z)| = o\left( \frac{t}{N^{1 - \varepsilon}}\right)  
\end{equation}
with overwhelming probability.

%%%%%%%%%%%%%%%%%%%%%%%%%%%%%%%%%%%%%%%%%%%%%%%%%%%%%%%%%%%%%%%%%%%%%%%%%%%%%%%%%%%%%%%%%%%%%%%%%%%%%%%%%%%%%%%%%%%%%%%%%%%%%%%%%%%%%%%%%%%%%%%%%%%%%%%%%%%%%%%%%%%%%%%%%%%%%%%%%%%%%%%%%%%%%%%%%%%%%%%%%%%%%%%%%%%%%%%%%%%%%%%%%%%%%%%%%%%%%%%%%%%%%%%%%%%%%%%%%%%%%%%%%%%%%%%%%%%%%%%%%%%%%%%%%%%%%%%%%%%%%%%%%%%%%%%%%%%%%%%%%%%%%%%%%%%%%%%%%%%%%%%%%%%%%%%%%%%%%%%%%%%%%%%%%%%%%%%%%%%%%%%%%%%%%%%%%%%%%%%%%%%%%%%%%%%%%%%%%%%%%%%%%%%%%%%%%%%%%%%%%%%%%%%%

\subsection{Proof of Theorem \ref{thm:Local law theorem}} 

In this section we provide the proof of Theorem \ref{thm:Local law theorem}. We will give begin the proof for generic initial starting positions of the Dyson Brownian motion, before specializing to the starting positions at the origin. Define the matrix \begin{equation}\label{eq:A:Matrix model}
    L_t=D-2\sqrt{t}A
\end{equation} where $A$ is drawn from the Gaussian Orthogonal/Unitary Ensemble and $D$ is an $N\times N$ deterministic diagonal matrix. Define the resolvent matrices  \begin{equation*}
    G_t(z):=\left( L_t-zI \right)^{-1}, \quad \text{and} \quad Q(z):=\left( D-zI \right)^{-1}.
\end{equation*} Next, we define the functions \begin{equation*}
    M_t^N(z)=-\frac{2}{N}\tr G_t(z), \quad \text{and} \quad S^N(z)=-\frac{2}{N}\tr Q(z).
\end{equation*} Fix $t,\eta>0$ and $z$ such that $\Im(z)\geq \eta$. Additionally, define the matrices \begin{equation*}
    G:= G_t(z),
\end{equation*} and \begin{equation*}
    Q:=Q\left(z-2t\EE M_t^N(z)\right).
\end{equation*} In particular $S^N\left(z-2t\EE M_t^N(z)\right)=-\frac{2}{N}\tr Q$. By the resolvent identity \eqref{eq:resolvent} \begin{align}\label{eq:A:Resolvent identity step}
    \EE M_t^N(z) - S^N\left(z-2t\EE M_t^N(z)\right) &= -\frac{2}{N}\left(\tr G_t -\tr Q_t \right) \\
                            &= -2\EE\frac{1}{N} \tr \left(G \Tilde{A} Q \right)+4t\EE M_t^N(z) \EE \frac{1}{N} \tr \left(GQ \right) \nonumber
\end{align} where $\Tilde{A}=2\sqrt{t}A$. We now consider the term \begin{equation}\label{eq:A:term for GIBP}
    -2\EE\frac{1}{N} \tr \left(G \Tilde{A} Q \right)=-\frac{2}{N}\sum_{i,j} Q_{ii}\EE \left[G_{ij}\Tilde{A}_{ji} \right].
\end{equation} A computation involving the resolvent identity \eqref{eq:resolvent} shows that \begin{equation*}
    \frac{\partial G_{kl}}{\partial A_{ij}}=\begin{cases}
        G_{ki}G_{ji}+G_{kj}G_{il},& \text{if } i\neq j,\\
        G_{ki}G_{jl}, & \text{if } i=j
    \end{cases}.
\end{equation*} Applying Gaussian integration by parts to \eqref{eq:A:term for GIBP} yields \begin{equation*}
    -2\EE\frac{1}{N} \tr \left(G \Tilde{A} Q \right)=\frac{-8}{N^2}\EE\sum_{i,j}Q_{ii}G_{ij}^2-\frac{4t}{N}\EE M_t^{N}(z)\tr (QG),
\end{equation*} which when combined with \eqref{eq:A:Resolvent identity step} gives \begin{align}\label{eq:A:Step after GIBP}
    \EE M_t^N(z) - S^N\left(z-2t\EE M_t^N(z)\right) &=\frac{-8}{N^2}\EE\sum_{i,j}Q_{ii}G_{ij}^2-\frac{4t}{N}\EE M_t^{N}(z)\tr (QG)\\
    &\quad+4t\EE M_t^N(z) \EE \frac{1}{N} \tr \left(GQ \right). \nonumber
\end{align} We now fix $z=E+i\eta\in\CC_+$. By the Ward identity 
\eqref{eq:ward} \begin{align*}
    \left|\frac{8}{N^2}\EE\sum_{i,j}Q_{ii}G_{ij}^2 \right|&\leq \EE \frac{8}{N^2}\sum_{j}|Q_{ii}|\sum_{j}|G_{ij}|^2\\
    &\leq \EE \frac{8}{N^2\eta}\sum_{i}|Q_{ii}|\Im G_{ii}\\
    &\leq \frac{8}{N\eta^3}.
\end{align*} For the difference $4t\EE M_t^N(z) \EE \frac{1}{N} \tr \left(GQ \right)-\frac{4t}{N}\EE M_t^{N}(z)\tr (QG)$, note that \begin{align*}
    \left|\frac{4t}{N}\tr \left(GQ \right) \right|&=\left| \frac{4t}{N}\sum_{i}Q_{ii G_{ii}} \right|\\
    &\leq \frac{4t}{\eta^2}.
\end{align*} It then follows from \eqref{eq:A:Concentration Conclusion} with $D$ equal to the zero matrix that\begin{align*}
    \EE & \left[\left| 4t\left(\EE M_t^N(z)\right) \frac{1}{N} \tr \left(GQ \right) -\frac{4t}{N} M_t^{N}(z)\tr (QG) \right| \right] \\
    &\qquad \leq \EE\left[\left|  M_t^N(z) \EE  -\EE M_t^{N}(z) \right| \left|\frac{4t}{N} \tr \left(GQ \right) \right| \right]\\
    &\qquad =o\left(\frac{4\max(t,t^2)}{N^{1-\eps}\eta^2} \right).
\end{align*} Thus, we conclude that \begin{equation}
    \EE M_t^N(z) - S^N\left(z-2t\EE M_t^N(z)\right)=O\left(\frac{4\max(t,t^2)}{N^{1-\eps}\eta^3} \right),
\end{equation} where $S^N(z)=\frac{2}{z}$ for all $N$. Let $M^\infty_t$ be defined as in Theorem \ref{thm:Local law theorem}, then \begin{equation*}
    M_t^\infty(z)-S^N\left(z-2t M_t^\infty(z)\right)=0.
\end{equation*} Note for each $z\in\CC_+$, $s_t=-\frac{1}{2}M_t^\infty(z)$, $\Tilde{s}_t=-\frac{1}{2}\EE M_t^N(z)$, and $s_0(z)=-\frac{1}{2}S^N(z)$ satisfy the conditions of Proposition \ref{Prop:A:stability}, (see Appendix) and hence it follows from Proposition \ref{Prop:A:stability} and \eqref{eq:Small t stability} (see Appendix) that \begin{equation}\label{eq:A:Expected value is close}
    \EE M^N_t(z)-M^\infty_t(z)=O\left(\frac{4^{1/3}\max(t,t^2)^{1/3}}{N^{1/3-\eps}\eta} \right).
\end{equation} Applying \eqref{eq:A:Concentration Conclusion} to \eqref{eq:A:Expected value is close} completes the proof of Theorem \ref{thm:Local law theorem}.

% The following result can be viewed as independent of Conjecture \ref{conjecture:concentration}. \begin{lemma}
%        Let $M_t(z)$ be a solution to the complex Burgers equation \begin{equation}
%            \begin{cases}
%                   \frac{\partial M_t(z)}{\partial t}=-2M_t(z)\frac{\partial M_t(z)}{\partial z}, t>0\\
%                   M_0(z)=\int_\RR\frac{2d\mu(x)}{z-x}
%            \end{cases}.
%        \end{equation} Let $G \subseteq \mathbb{H}$ be a compact subset such that $\Im(z)\geq\eta>0$ for all $z\in G$ and let $M_{t,\ee}(z)$ be a a holomorphic function on the upper half plane such that on $G$ \begin{equation}
%            \begin{cases}
%                   \frac{\partial M_{t,\ee}(z)}{\partial t}=-2M_{t,\ee}(z)\frac{\partial M_{t,\ee}(z)}{\partial z}+O(\ee), t>0\\
%                   M_{0,\ee}(z)=\int_\RR\frac{2d\mu(x)}{z-x}+O(\eps)
%            \end{cases}.
%        \end{equation} Then for fixed $t\in [0,T]$\begin{equation}
%            \sup_{z\in G}\left|M_t(z)-M_{t,\ee}(z)\right|=O(\ee)+O(\delta).
%        \end{equation}
% \end{lemma}

% \begin{proof}
   
% \end{proof}

\subsection{Extension to uniform bound over $[0,T]$}\label{section: net to all t}

In this section we outline how to extend Theorem \ref{thm:Local law theorem} uniformly in $t \in [0,T]$. This relies on the continuity of DBM.  Without loss of generality, we work with the interval $[0, 1]$ instead of the interval $[0,T]$. Let us consider a partition of the time interval $[0,1]$ into a uniform partition with $t_k=\frac{k}{n}$, $k =0, 1, \ldots, n.$ The intervals of this partition are all equally-sized and their lengths are equal to $\frac{1}{n}.$

Let us consider $t \in (t_1, t_2)$ an intermediate time.
We have that 
\begin{align}\label{eq:A:M triangle}
&\sup_{z \in G}|M_t^{\infty}(z)-M^N_t(z)|\nonumber\\
&\leq \sup_{z \in G}|M_{t}^{\infty}(z)-M^{\infty}_{t_1}(z)| + \sup_{z \in G}|M_{t_1}^{\infty}(z)-M^{N}_{t_1}(z)|+ \sup_{z \in G}|M_{t_1}^{N}(z)-M^{N}_{t}(z)|, 
\end{align}
with $G$ being a particular subset of the complement of the hull as in the previous section.
The first term can be controlled from the Burgers equation as the solution is locally Lipschitz in time. 

For the second term of the right hand side of \eqref{eq:A:M triangle}, we have that from Theorem \ref{thm:Local law theorem}, for any $\epsilon >0$

\begin{equation}
\sup_{z \in G}|M_{t_1}^{\infty}(z)-M^N_{t_1}(z)|=O_\eps\left(\frac{1}{N^{1/3-\epsilon}}\right)
\end{equation}
with overwhelming probability, that is with probability at least $1- e^{-cN}$, for some constant $c$.

By a union bound for any $t_j$, $j=1,\ldots, n$ in the net, we have that

\begin{align}
&\mathbb{P}\left( \bigcup_{t_i}|M^{\infty}_{t_i}(z)-M_{t_i}^N(z)|=\Omega\left(\frac{C}{N^{1/3-\epsilon}}\right)\right)\nonumber\\
&\leq \sum_{i=1}^n \mathbb{P}\left(| M^{\infty}_{t_i}(z)-M_{t_i}^N(z)|=\Omega\left(\frac{C}{N^{1/3-\epsilon}}\right)\right)\nonumber\\
&\leq n e^{-CN},
\end{align}
where $g=\Omega(f)$ means $\frac{g(x)}{f(x)}$, as $x \to \infty.$

For the third term of the right hand side of \eqref{eq:A:M triangle}, using the notation $\tilde{\eta}^i_t=z-\lambda^i_t$, for $i=1, 2,\ldots, N$, we have that

\begin{equation}
|M^{N}_{t_1}(z)-M_t^N(z)| \leq \frac{2}{N}\sum_{i=1}^N\frac{|\lambda^i_t-\lambda^i_{t_1}|}{|\tilde{\eta}^i_{t_1}\tilde{\eta}^i_t|} \leq \frac{\tilde{C}|t-t_1|^{1/2-\epsilon}}{\text{Im}(z_0)^2},
\end{equation}
where we have used the regularity of the Dyson Brownian Motion driver (\cite{NualartPerez}) and the bound $|\Tilde{\eta}^i_t |\geq |\Im(z)| \geq |\Im(z_0)|$ where $z_0 \in \mathbb{H}$ such that $\Im(z_0)\leq \min_{z\in G}(\Im(z))$.

Using the notation $\hat{C}=\frac{\tilde{C}}{\text{Im}(z_0)^2}$, if we want the error to not accumulate in our net we need 

$$\hat{C}\frac{1}{n^{1/2-\epsilon}} \leq \frac{C}{N^{1/3-\epsilon}}.$$ Thus, for our partition of the time interval we have 

$$n > \frac{\hat{C}^2(N^{(1/3-\epsilon)})^2}{C^2},$$ for $\hat{C}$ and $C$ some constants. It then follows from \eqref{eq:A:M triangle}, that \begin{equation}\label{eq:A:Uniform local}
    \sup_{t\in [0,1],\ z\in G}\left|  M_t^N(z)-M_t^\infty(z) \right|=O\left( \frac{1}{N^{\frac{1}{3}-\eps}} \right).
\end{equation}

\section{Proof of Theorem \ref{thm: Main Result}} In this section we will complete the proof of Theorem \ref{thm: Main Result}. Fix $\eps > 0$. Let $G$ be a suitable compact subset of $\CC_+$ and let $\Tilde{G}$ be a compact subset of $\CC_+$ such that $g_t^N(G)\subseteq \Tilde{G}$ with overwhelming probability (see \eqref{eq:A:set containment} for the existence of such a $\Tilde{G}$). Begin by defining $\eta:=\min_{z\in\Tilde{G}}(\Im z) > 0$. Note that \begin{align}\label{eq:A:g difference}
    |g_t^N(z)-g_t^\infty(z)|&=\left|\int_0^t M_s^N(g_s^N(z))-M_s^\infty(g_s^\infty(z))ds \right|\\
        &\leq\left|\int_0^t M_s^N(g_s^N(z))-M_s^\infty(g_s^N(z))ds \right| \\
        &\quad \quad +\left|\int_0^t M_s^\infty(g_s^N(z))-M_s^\infty(g_s^\infty(z))ds \right|.\nonumber
\end{align} For the term $ M_s^N(g_s^N(z))-M_s^\infty(g_s^N(z))$, observe that from Theorem \ref{thm:Local law theorem} \begin{equation*}
    \sup_{z\in \Tilde{G}}\left|  M_s^N(z)-M_s^\infty(z) \right|=O\left( \frac{4T^2}{N^{\frac{1}{3}-\eps}} \right),
\end{equation*} for fixed $s\in [0,T]$ with overwhelming probability. From the argument in Section \ref{section: net to all t} this can be extended to \begin{equation}\label{eq:A: Local law uniformly bounded}
    \sup_{s\in [0,T],\ z\in \Tilde{G}}\left|  M_s^N(z)-M_s^\infty(z) \right|=O\left( \frac{4T^2}{N^{\frac{1}{3}-\eps}} \right).
\end{equation} For the term $M_s^\infty(g_s^N(z))-M_s^\infty(g_s^\infty(z))$, note that $M_s^\infty$ is at most $\frac{2}{\eta^2}$-Lipschitz on $\Tilde{G}$, and hence \begin{equation}\label{eq:A:Lipschitz bound for Gronwall's}
    \left|M_s^\infty(g_s^N(z))-M_s^\infty(g_s^\infty(z))\right|\leq \frac{2}{\eta^2}|g_t^N(z)-g_t^\infty(z)|.
\end{equation} From \eqref{eq:A:g difference}, \eqref{eq:A: Local law uniformly bounded}, and \eqref{eq:A:Lipschitz bound for Gronwall's}, we conclude that \begin{equation*}
     |g_t^N(z)-g_t^\infty(z)|\leq O\left( \frac{4T^2}{N^{\frac{1}{3}-\eps}} \right)+\int_0^t\frac{2}{\eta^2}|g_s^N(z)-g_s^\infty(z)|ds.
\end{equation*} Theorem \ref{thm: Main Result} then follows from Gr\"onwall's iequality.

\newpage

\appendix{\huge{\textbf{}}}

\section{Stability} 

The following is essentially a result of O'Rourke and Vu (\cite{SeanVu}). We provide the details for the time change for the convenience of the reader.
\begin{proposition}[Stability for positive time]\label{Prop:A:stability}
   Let $t>0$ and $z,s_t,\Tilde{s}_t$ be elements of the upper half-plane such that \begin{equation}\label{eq:A:exact self-consist solution}
       s_t=s_0(z+4 t s_t),
   \end{equation} and \begin{equation}\label{eq:A:approx self-consist solution}
       \Tilde{s}_t=s_0(z+4 t \Tilde{s}_t)+O(\eps),
   \end{equation} with $s_0(z)=\int_\RR\frac{d\mu_0}{x-z}$ for some compactly supported probability measure $\mu_0$, some $R\geq z$, and small $\eps>0$. Additionally assume there exists $\eta>0$ such that $\Im(z)\geq \eta$. Then $s,s'=O(1)$ and \begin{equation}\label{eq:A:arb stability result}
       s_t=\Tilde{s}_t+O\left(\frac{\eps^{1/3}}{(4t)^{2/3}\eta} \right).
   \end{equation}
\end{proposition}

\begin{proof}
    Showing $s_t,\Tilde{s}_t=O(1)$ requires no change from O'Rourke and Vu. Let $w_t=z+4t s_t$ and $\Tilde{w}_t=z+4t \Tilde{s}_t$. We aim now to show $|w_t-\Tilde{w}_t|$ is sufficiently small It follows from \eqref{eq:A:exact self-consist solution} and \eqref{eq:A:approx self-consist solution} that\begin{equation}\label{eq:A:diffference as rational}
        s_0(w_t)-s_0(\Tilde{w}_t)=\frac{w_t-\Tilde{w}_t}{4t}+O(\eps).
    \end{equation} It additionally follows from the definition of $s_0$ that \begin{equation*}
        s_0(w_t)-s_0(\Tilde{w}_t)=(w_t-\Tilde{w}_t)\int_{\RR}\frac{d\mu_0(z)}{(x-w_t)(x-\Tilde{w}_t)},
    \end{equation*} which when combined with \eqref{eq:A:diffference as rational} yields\begin{equation}\label{eq:A:integral of difference equation}
        \int_{\RR}\frac{d\mu_0(z)}{(x-w_t)(x-\Tilde{w}_t)}=\frac{1}{4t}+O\left(\frac{\eps}{|w_t-\Tilde{w}_t|} \right).
    \end{equation} On the other hand \begin{align}
        \Im (s_t) &=\Im (s_0(w_t)) \nonumber\\
                  &= \Im(w_t)\int_{\RR}\frac{d\mu_0(x)}{|x-w_t|^2},
    \end{align} and rearranging yields\begin{equation}\label{eq:A:integral of w bound}
        \int_{\RR}\frac{d\mu_0(x)}{|x-w_t|^2}=\leq \frac{1}{4t}.
    \end{equation} An identical argument yields \begin{equation}\label{eq:A:integral of tilde w bound}
        \int_{\RR}\frac{d\mu_0(x)}{|x-\Tilde{w}_t|^2}=\leq \frac{1}{4t}+O\left(\frac{\eps}{\eta} \right).
    \end{equation}

 From the arithmetic mean-geometric mean inequality, we have
$$
\left|\frac{1}{(x-w_t)\left(x-\tilde{w}_t\right)}\right| \leq \frac{1}{2} \frac{1}{|x-w_t|^2}+\frac{1}{2} \frac{1}{\left|x-\tilde{w}_t\right|^2} .
$$
Since $w_t \neq \tilde{w}_t$, it follows that
$$
\left|\operatorname{Re}\left[\frac{1}{(x-w_t)\left(x-\tilde{w}_t\right)}\right]\right|=(1-\delta)\left(\frac{1}{2} \frac{1}{|x-w_t|^2}+\frac{1}{2} \frac{1}{\left|x-\tilde{w}_t\right|^2}\right)
$$
for some $\delta>0$. Then we have
$$
|x-w_t|=(1+O(\delta))\left|x-\tilde{w}_t\right| .
$$
and
$$
\angle\left(x-w_t, x-\tilde{w}_t\right)=O\left(\delta^{1 / 2}\right) .
$$
Since $x, w_t, \tilde{w}_t=O(1)$, we obtain $w_t-\tilde{w}_t=O\left(\delta^{1 / 2}\right)$.
We obtain that
 $\operatorname{Re}\left[\frac{1}{(x-w_t)\left(x-\tilde{w}_t\right)}\right] \leq\left(1-C\left|w_t-\tilde{w}_t\right|^2\right)\left(\frac{1}{2} \frac{1}{|x-w_t|^2}+\frac{1}{2} \frac{1}{\left|x-\tilde{w}_t\right|^2}\right)$ for some $C>0$, and hence
$$
\operatorname{Re} \int_{\mathbb{R}} \frac{d \mu(x)}{(x-w_t)\left(x-\tilde{w}_t\right)} \leq\left(1-C\left|w_t-\tilde{w}_t\right|^2\right)\left(\frac{1}{4t}+O\left(\frac{\varepsilon}{\eta}\right)\right).
$$

We have that 
\begin{equation}
\left(1-C|w_t-\tilde{w}_t|^2\right)\left(\frac{1}{4t}+O\left(\frac{\epsilon}{\eta}\right)\right)=\frac{1}{4t}+O\left( \frac{\epsilon}{|w_t-\tilde{w}_t|}\right).
\end{equation}
Then, 
\begin{equation}
\frac{1}{4t}+O\left(\frac{\epsilon}{\eta}\right)-\frac{c}{4t}|w_t-\tilde{w}_t|^2-C|w_t-\tilde{w}_t|^2O\left(\frac{\epsilon}{\eta}\right)=\frac{1}{4t}+O\left(\frac{\epsilon}{|w_t-\tilde{w}_t|}\right).
\end{equation}

Furthermore, we obtain
\begin{equation}
O\left(\epsilon\right)=|w_t-\tilde{w}_t|O\left(\frac{\epsilon}{\eta}\right)-\frac{C}{4t}|w_t-\tilde{w}_t|^3-C|w_t-\tilde{w}_t|^3O\left(\frac{\epsilon}{\eta}\right).
\end{equation}

Using that the first and the third term are bounded we obtain
\begin{equation}
O(\epsilon)+O\left(\frac{\epsilon}{\eta}\right)=|w_t-\tilde{w}_t|^3\frac{C}{4t}.
\end{equation}

Thus, we have that 
\begin{equation}
|w_t-\tilde{w}_t|=O\left(\left(\frac{4t\epsilon}{\eta}\right)^{1/3}\right),
\end{equation}
and
\begin{equation}
|s_t-\tilde{s}_t|=O\left(\frac{\epsilon^{1/3}}{(4t)^{2/3}\eta^{1/3}}\right).
\end{equation}
\end{proof} 

For small $t$  the following observation is useful. Fix $\eta>0$, then $s_0$ is Lipschitz with Lipschitz constant at most $\frac{1}{\eta^2
}$ on $\{z\ :\ \Im(z) \geq \eta  \}$. Thus for $s_t$ and $\Tilde{s}_t$ as in \eqref{eq:A:exact self-consist solution} and \eqref{eq:A:approx self-consist solution} one has \begin{equation}\label{eq:Small t stability}
    s_t-\Tilde{s}_t=\frac{4t}{\eta^2}\left( s_t -\Tilde{s}_t \right) + O\left( \eps \right).
\end{equation}

\bibliography{MultipleSLEfinal2.bib}
\bibliographystyle{abbrv}

\end{document}